\documentclass[hidelinks,onefignum,onetabnum]{siamart220329}

\usepackage{lipsum}
\usepackage{amsfonts}
\usepackage{graphicx}
\usepackage{epstopdf}
\usepackage{algorithmic}
\ifpdf
  \DeclareGraphicsExtensions{.eps,.pdf,.png,.jpg}
\else
  \DeclareGraphicsExtensions{.eps}
\fi

\newsiamremark{remark}{Remark}
\newsiamremark{hypothesis}{Hypothesis}
\crefname{hypothesis}{Hypothesis}{Hypotheses}
\newsiamthm{claim}{Claim}

\headers{Solve high-order proximal operator }{Jingyu Gao and Xiurui Geng}

\title{A linearly convergent method for solving high-order proximal operator\thanks{Submitted to the editors DATE.
}}

\author{Jingyu Gao\thanks{Aerospace Information Research Institute, Chinese    Academy of Sciences; School of Electronic, Electrical and Communication Engineering, University of Chinese Academy of Sciences; Key Laboratory of Technology in Geo-Spatial Information Processing and Application System, Chinese Academy of Sciences 
  .}
\and Xiurui Geng\thanks{Aerospace Information Research Institute, Chinese    Academy of Sciences; School of Electronic, Electrical and Communication Engineering, University of Chinese Academy of Sciences; Key Laboratory of Technology in Geo-Spatial Information Processing and Application System, Chinese Academy of Sciences 
  (\email{genxr@sina.com.cn}).}
}

\usepackage{amsopn}

\ifpdf
\hypersetup{
  pdftitle={A linearly convergent method for solving high-order proximal operator},
  pdfauthor={Jingyu Gao and Xiurui Geng}
}
\fi

\begin{document}

\maketitle

\begin{abstract}
  Recently, various high-order methods have been developed to solve the convex optimization problem. The auxiliary problem of these methods shares the general form that is the same as the high-order proximal operator proposed by  Nesterov. In this paper, we present a linearly convergent method to solve the high-order proximal operator based on the classical proximal operator. In addition, some experiments are performed to demonstrate the performance of the proposed method.
\end{abstract}



\section{Introduction}
In this work, we consider the optimization problems with the regularization term which is a power of a norm
\begin{equation}\label{equ1}
  \mathop {\min }\limits_{x \in \mathbb{R} ^{n}  } \{ F(x) = f(x) + \frac{\sigma }{{1 + p}}\left\lVert x-c\right\rVert^{p+1} \},
\end{equation}
where $ f:\mathbb{R} ^{n} \to \mathbb{R} $ is a closed proper convex function, and $p\geq 1, \sigma >0, c \in \mathbb{R} ^{n}$. Numerous instances of \cref{equ1} can be found in the literature. For example, the subproblem of the proximal point algorithm (PPA) \cite{moreau1965proximite} has the same structure of \cref{equ1} with $p=1$. The PPA is a fundamental method
in optimization theory, serving as the progenitor to a number of famous methods such as the augmented Lagrangian method (ALM) \cite{hestenes1969multiplier, powell1969method}, the alternating direction method of multipliers (ADMM) \cite{boyd2011distributed, hong2017linear}, the Douglas-Rachford
operator splitting method (DRSM) \cite{lions1979splitting, combettes2007douglas}, and so on. In the case of $p=1$, $F(x)$ is strongly convex, and the gradient descent method for \cref{equ1} has the linear convergence rate, assuming the gradient of $f$ is Lipschitz continuous. Thus, the computational expense of solving the regularized problem \cref{equ1} with $p=1$ is economical.

For $p=2$, the optimization problem \cref{equ1} emerges as the subproblem of cubic regularization method (CRM) \cite{nesterov2006cubic}. In this case, $f$ is a quadratic function, i.e.
\begin{equation}\label{equ2}
  f(x)=\frac{1}{2} x^{T}Ax+b^{T}x,
\end{equation}
where $A\succeq 0$ and $b\in \mathbb{R} ^{n}$. The cubic regularization method (CRM) is a variant of the classical Newton method for solving the unconstrained optimization problems \cite{nesterov2006cubic}. Under some mild assumptions, the CRM converges to the second order critical points, and has the convergence rate $O ({k^{{{ - 2} \mathord{\left/
{\vphantom {{ - 2} 3}} \right.
\kern-\nulldelimiterspace} 3}}})$ for the norms of the gradients. 
For $p=3$, the optimization problem \cref{equ1} appears to be the auxiliary problem in the third-order tensor method \cite{nesterov2021implementable}. Moreover, the subproblem of the new second-order method based on quartic regulation \cite{nesterov2022quartic} has the same form of \cref{equ1}.

Recently, Nesterov  has extended the PPA formulation to encompass arbitrary order $p \geq 1$ \cite{nesterov2021inexact, nesterov2021inexact2}. The iterative scheme of the  high-order PPA can be expressed as
\begin{equation}\label{equ3}
  x^{k+1} = \mathop {\arg \min }\limits_{x \in \mathbb{R} ^{n} } \{  f(x) + \frac{\sigma }{{1 + p}}\left\lVert x-x^{k}\right\rVert^{p+1} \}, \quad \quad k\geq 0.
\end{equation}
Assume that the subproblem \cref{equ3} can be solved inexactly at each iteration, the convergence rate of the high-order PPA is that
$f(x^{k})-\min f \leq O (1/k^{p})$, and the accelerated version of the high-order PPA converges as  $O (1/k^{p+1})$.

In general, the regularization problem \cref{equ1} with $p>1$ tends to emerge within higher-order methods, which is proved to converge much faster than the first-order method in theory. However, the regularization problem \cref{equ1} with $p>1$ is more complicated than the case of $p=1$. Thus, the problem of exploiting the relationship between the classical proximal operator and the $p$th-order proximal operator ($p>1$) is meaningful. Specifically, this problem can be summarized as follows: Can we design an efficient method  to solve the $p$th-order proximal operator ($p>1$) based on the classical proximal operator?
In this paper, we  provide an affirmation to this problem by proposing a linearly convergent method. 

The paper is organized as follows. In \cref{sec:pre}, we introduce some basic notations and properties. In \cref{sec:dual}, we derive the dual problem of the regularization problem \cref{equ1}. Our main results are presented in \cref{main}. A linearly convergent method is developed to solve the high-order proximal operator. In \cref{sec:special}, we consider the regularization problem \cref{equ1} with the special case $p=2$, and it can be transformed into solving a one-dimensional monotonic continuity equation.  In \cref{sec:experiments}, we conduct some numerical experiments to demonstrate the performance of the method proposed in \cref{main}.

\section{Preliminaries}
\label{sec:pre}
Given a convex function $f$, its conjugate function $f^{\ast }$ is defined as 
\begin{equation}\label{equ4}
  f^{\ast }(y)=\mathop {\sup }\limits_x \left\{x^{T}y-f(x)\right\}  . \nonumber
\end{equation}
For example, when $f(x) = I_{\left\{b\right\}} (x)$, where $I_{\left\{b\right\}} (x)$ is the characteristic function of a single point set $ \left\{b\right\} $, it holds that
\begin{equation}\label{equ5}
  f^{\ast }(y)=b^{T}y. 
\end{equation}
The following lemma gives an important property of the conjugate function \cite{boyd2004convex}.
\begin{lemma}\label{newlemma_1}
  Suppose $f\left(x\right) $ is a closed proper convex function, and $f^{\ast }\left(y\right) $ is the corresponding conjugate function. Then,
  \begin{equation}\label{equ6}
      y\in \partial f\left(x\right) \Leftrightarrow x \in \partial f^{\ast }\left(y\right).
  \end{equation}
\end{lemma}
Following the definition in \cite{nesterov2021inexact}, the $p$th-order proximal operator is defined as 
\begin{equation}\label{equ7}
  \mathsf{P}\mathrm{rox}_{ f/\sigma}^{p}\left(x\right)  : = \arg \min \left\{ F(y)=f\left(y\right) +\frac{\sigma}{p+1} \left\lVert y-x\right\rVert ^{p+1}\vert \; y\in \mathbb{R} ^{n}\right\},    
\end{equation}
where $\sigma>0$ and $p\geq1$. For the sake of the notation, we denote $\mathsf{P}\mathrm{rox}_{ f/\sigma}$ as the classical proximal operator, i.e. $p=1$.  First, we prove that $\mathsf{P}\mathrm{rox}_{ f/\lambda}^{p}$ is well defined if $f$ is a closed proper convex function.
\begin{lemma}\label{newlemma_2}
  Suppose $f\left(x\right) $ is a closed proper convex function. Then, for $\forall x \in \mathbb{R} ^{n}$, $\mathsf{P}\mathrm{rox}_{ f/\lambda}^{p}\left(x\right)$ uniquely exists.
\end{lemma}
\begin{proof}
  Since $f$ is a closed proper convex function, $\mathbf{ri}\ \mathbf{dom}f$ is nonempty \cite{ryu2022large}. Taking $y_0 \in \mathbf{ri}\ \mathbf{dom}f$, and we have 
  \begin{equation}\label{equ8}
    f(y)+\frac{\sigma}{p+1}\left\lVert y-x\right\rVert^{p+1}   \geq f(y_0)+g^{T}(y-y_0)+\frac{\sigma}{p+1}\left\lVert y-x\right\rVert^{p+1},
  \end{equation}
  where $g \in \partial f(y_0)$. In view of \cref{equ8}, it can be easily obtained that
  \begin{equation}\label{equ9}
    \lim_{\left\lVert y\right\rVert  \to \infty}  F(y) = \infty.
  \end{equation}
  Thus, $F(y^{k})\to \inf F(y)$ implies that $\left\{{y^{k}}\right\}_{k\geq 0}$ is bounded. Due to the lower semicontinuity of $F$, $F(y^{\ast}) \leq \inf F(y) $, where $y^{k_j} \to y^{\ast}$ and $\left\{y^{k_j}\right\}$ is a convergent subsequence of $\left\{{y^{k}}\right\}_{k\geq 0}$. Therefore, $F(y^{\ast}) = \inf F(y) $. Proof of existence completed.

  Assume that  $F(y_1)=F(y_2)= \inf F(y)$, then we have 
  \begin{equation}\label{equ10}
    \left\lVert x-y_1\right\rVert^{p-1}(x-y_1) \in \frac{1}{\sigma} \partial f(y_1) .
  \end{equation}
  \begin{equation}\label{equ11}
    \left\lVert x-y_2\right\rVert^{p-1}(x-y_2) \in \frac{1}{\sigma} \partial f(y_2) .
  \end{equation}
  Since $f$ is convex, it holds that 
  \begin{equation}\label{equ12}
    (y_1-y_2)^{T}\left\{\left\lVert x-y_1\right\rVert^{p-1} (x-y_1)-\left\lVert x-y_2\right\rVert^{p-1} (x-y_2)\right\} \geq 0.
  \end{equation}
  Let $u_1 = x-y_1$ and $u_2 = x-y_2$, \cref{equ12} turns into 
  \begin{equation}\label{equ13}
    (u_1-u_2)^{T}(\left\lVert u_1\right\rVert^{p-1}u_1 -\left\lVert u_2\right\rVert^{p-1}u_2) \leq 0.
  \end{equation}
  Note that $g(u)=\frac{1}{p+1} \left\lVert u\right\rVert ^{p+1}$ is convex, and $\nabla g(u) = \left\lVert u\right\rVert ^{p-1}u$, we have $(u_1-u_2)^{T}(\left\lVert u_1\right\rVert^{p-1}u_1 -\left\lVert u_2\right\rVert^{p-1}u_2) \geq 0$. Therefore, $(u_1-u_2)^{T}(\left\lVert u_1\right\rVert^{p-1}u_1 -\left\lVert u_2\right\rVert^{p-1}u_2) = 0$ holds. This implies that $u_1=u_2$, i.e. $y_1=y_2$. Proof of uniqueness completed.
\end{proof}

Finally, we introduce an important function which is frequently used in the following sections
\begin{equation}\label{equ14}
  i_p(x) = \left\{ {\begin{array}{*{20}{r}}
    \frac{x}{\left\lVert x\right\rVert ^{1-\frac{1}{p} }}\qquad x\neq 0\\
    0  \quad  \qquad x=0
    \end{array}} \right. .
\end{equation}
Note that $i_p(x)$ is the gradient of the convex function $\frac{1}{1+\frac{1}{p} } \left\lVert x\right\rVert ^{1+\frac{1}{p}}$.

\section{Dual pth-order PPA}
\label{sec:dual}

In this section, we derive the dual problem of the regularization problem \cref{equ1}. Denote $f_1(x) = f(x)$ and $f_2(x)= \frac{\sigma }{{1 + p}}\left\lVert x-c\right\rVert^{p+1}$. The dual problem of \cref{equ1} is
\begin{equation}\label{equ15}
  \mathop {\min }\limits_{\lambda \in \mathbb{R} ^{n}  } \{ f_1^{\ast}(\lambda)+ f_2^{\ast}(-\lambda)\}.
\end{equation}
In the following lemma, we will see that $f_2^{\ast}(\lambda)$ has an explicit expression. 
\begin{lemma}\label{newlemma_3}
  If $f_2(x) = \frac{\sigma }{{1 + p}}\left\lVert x-c\right\rVert^{p+1}$, then 
  \begin{equation}\label{equ16}
    f_2^{\ast}(\lambda) = \lambda^{T}c+\frac{\sigma^{-\frac{1}{p} }}{1+\frac{1}{p} }\left\lVert \lambda\right\rVert^{1+\frac{1}{p} }  .
  \end{equation}
\end{lemma}
\begin{proof}
  According to definition of the conjugate function, we have 
  \begin{equation}\label{equ17}
    f_2^{\ast}(\lambda) = \mathop {\sup }\limits_y  \{\lambda^{T}y-\frac{\sigma}{p+1} \left\lVert y-c\right\rVert^{p+1}\}.
  \end{equation}
  Due to the first-order optimal condition, the optimal solution of \cref{equ17} $y^{\ast}$ satisfies that 
  \begin{equation}\label{equ18}
    \lambda = \sigma \left\lVert y^{\ast}-c\right\rVert^{p-1}(y^{\ast}-c) .
  \end{equation}
  Thus, $\left\lVert y^{\ast}-c\right\rVert = \sigma^{-\frac{1}{p} }\left\lVert \lambda \right\rVert ^{\frac{1}{p} }$. Next, we discuss the two cases of $\lambda=0$ and $\lambda \neq 0$ respectively.

  Case 1 ($\lambda=0$): According to \cref{equ18}, $y^{\ast} = c$,  and $f_2^{\ast}(0) = 0$.

  Case 2 ($\lambda \neq 0$): According to \cref{equ18}, $y^{\ast} = c+\sigma^{-\frac{1}{p} }\frac{\lambda}{\left\lVert \lambda\right\rVert^{1-\frac{1}{p} } } $, and 
  \begin{equation}\label{equ19}
    \begin{aligned}
      f_2^{\ast}(\lambda) &= \lambda^{T}y^{\ast}-\frac{\sigma}{p+1} \left\lVert y^{\ast}-c\right\rVert^{p+1} \\
      &= \lambda^{T}c+\frac{\sigma^{-\frac{1}{p} }}{1+\frac{1}{p} }\left\lVert \lambda\right\rVert ^{1+\frac{1}{p} } 
      \end{aligned}.
  \end{equation}
\end{proof}

Therefore, the dual problem can be expressed as 
\begin{equation}\label{equ20}
  \mathop {\min }\limits_{\lambda \in \mathbb{R} ^{n}  } \{ f_1^{\ast}(\lambda)-\lambda^{T}c+\frac{\sigma^{-\frac{1}{p} }}{1+\frac{1}{p} }\left\lVert \lambda\right\rVert ^{1+\frac{1}{p} } \}.
\end{equation}
Using the notation \cref{equ14}, the optimal solution $\lambda^{\ast}$ of \cref{equ20} satisfies $c-\sigma^{-\frac{1}{p} }i_p(\lambda^{\ast}) \in \partial f_1^{\ast}(\lambda^{\ast})$. According to \cref{newlemma_1}, it holds that 
\begin{equation}\label{equ21}
  \lambda^{\ast} \in \partial f_1(c-\sigma^{-\frac{1}{p} }i_p(\lambda^{\ast})).
\end{equation}
After some simple calculation, it also holds that
\begin{equation}\label{equ22}
  -\lambda^{\ast} \in \partial f_2(c-\sigma^{-\frac{1}{p} }i_p(\lambda^{\ast})).
\end{equation}
Thus, $0\in \partial f_1(c-\sigma^{-\frac{1}{p} }i_p(\lambda^{\ast}))+\partial f_2(c-\sigma^{-\frac{1}{p} }i_p(\lambda^{\ast}))$, and $\mathsf{P}\mathrm{rox}_{ f/\sigma}^{p}\left(c\right) = c-\sigma^{-\frac{1}{p} }i_p(\lambda^{\ast})$ ($\mathsf{P}\mathrm{rox}_{ f/\sigma}^{p}$ is well defined, see in \cref{newlemma_2}). The detailed procedure of the dual $p$th-order PPA is given in \cref{algorithm1}.
\begin{algorithm} 
  \caption{Dual $p$th-order PPA}
  \label{algorithm1}
  \begin{algorithmic}[1]
  \STATE{Require: $x_0, K$}
  \FOR{$k = 0,1,2,\ldots K $}
  \STATE\label{l1}{$\lambda_{k+1}=\mathop {\min }\limits_{\lambda \in \mathbb{R} ^{n}  } \{ f^{\ast}(\lambda)-\lambda^{T}x_k+\frac{\sigma^{-\frac{1}{p} }}{1+\frac{1}{p} }\left\lVert \lambda\right\rVert ^{1+\frac{1}{p} } \}$}
  \STATE\label{l2}{$x_{k+1}=x_k-\sigma^{-\frac{1}{p} }i_p(\lambda^{k+1})$}
  \ENDFOR
  \RETURN{$x_{K}=\mathsf{P}\mathrm{rox}_{t_{K} f}\left(c\right)$}
  \end{algorithmic}
\end{algorithm}
\begin{remark}
  The regularization term in the dual problem of \cref{equ1} is $\left\lVert \lambda\right\rVert ^{1+\frac{1}{p} }$, and its power is less than 2. Therefore, the dual problem \cref{equ19} can be easier  to be solved compared to the prime problem \cref{equ1}.
\end{remark}

\section{Main method}
\label{main}
In this section, we focus on the dual problem \cref{equ20}, which shares the general form 
\begin{equation}\label{equ23}
  \mathop {\min }\limits_{\lambda \in \mathbb{R} ^{n}  } \{ g(\lambda)+\frac{\mu}{1+\frac{1}{p} }\left\lVert \lambda\right\rVert ^{1+\frac{1}{p} } \},
\end{equation}
where $\mu>0$ and $g$ is a convex function. Denote by $\lambda^{\ast }$ the optimal solution of \cref{equ23}. The optimal condition for \cref{equ23} can be expressed in the form of the variational inequality
\begin{equation}\label{equ24}
  (\lambda-\lambda^{\ast})^{T}\left\{h^{\ast}+\mu i_p(y^{\ast})\right\} \geq 0,  \quad \forall \lambda \in \mathbb{R} ^{n},
\end{equation}
where $h^{\ast} \in \partial g(\lambda^{\ast})$.

Our method for \cref{equ23} is the following iteration scheme.

\begin{algorithm} 
  \caption{Fixed point iteration for the \cref{equ23}}
  \label{algorithm2}
  \begin{algorithmic}[1]
  \STATE{Require: $\lambda^{0} \in \mathbb{R}^{n}, \; K$}
  \WHILE{$k = 1,2,...,K$}
     \IF{$\lambda^{k} \neq 0$}
        \STATE $t^{k} = \left\lVert \lambda^{k}\right\rVert ^{\frac{1}{p} -1}$ ; 
     \ELSE
       \STATE $t^{k} =0$;
     \ENDIF 
     \STATE $\lambda ^{k+1} = \arg  \min \{g\left(\lambda\right)+ \frac{\mu }{2 }t_k\left\lVert \lambda\right\rVert ^{2 } \vert \; \lambda\in\mathbb{R}^{n}   \} $
  \ENDWHILE
  \RETURN $\lambda^{K}$
  \end{algorithmic}
\end{algorithm}

Next, we discuss the convergence of \cref{algorithm2} in two cases.

 Case 1: $\lambda^{\ast }=0$
 \begin{theorem}\label{newTheorem1}
  Suppose $\lambda^{\ast }=0$ and $ \left\lVert \lambda^{k}\right\rVert >0$. Then, $\lambda^{k+1}$ generated by \cref{algorithm2} satisfies  that $\lambda^{k+1} = \lambda^{\ast }$.
\end{theorem}  
\begin{proof}
  Since $\lambda^{\ast }=0$, the optimal condition for \cref{equ23} can be simplified as 
  \begin{equation}\label{equ25}
    (\lambda-\lambda^{\ast})^{T}h^{\ast} \geq 0, \quad \forall \lambda \in \mathbb{R}^{n},
 \end{equation} 
 where $h^{\ast} \in \partial g(\lambda^{\ast})$. Setting $\lambda$ in \cref{equ25} as $\lambda^{k+1}$, we have
 \begin{equation}\label{equ26}
  (\lambda^{k+1}-\lambda^{\ast })^{T}h^{\ast} \geq 0.
\end{equation} 
 Using the optimality condition for the step 8 of \cref{algorithm2} yields
 \begin{equation}\label{equ27}
  (\lambda-\lambda^{k+1 })^{T}\left\{h^{k+1}+\mu \frac{\lambda^{k+1}}{\left\lVert \lambda^{k}\right\rVert^{1-\frac{1}{p} } }  \right\} \geq 0, \quad \forall y \in \mathcal{Y}.
 \end{equation} 
  where $h^{k+1} \in \partial g(\lambda^{k+1})$. Setting $\lambda^{\ast }$ in \cref{equ27} as $\lambda$, we have
 \begin{equation}\label{equ28}
  (\lambda^{\ast }-\lambda^{k+1 })^{T}\left\{h^{k+1}+\mu \frac{\lambda^{k+1}}{\left\lVert \lambda^{k}\right\rVert^{1-\frac{1}{p} } }  \right\} \geq 0.
 \end{equation} 
 Considering that $\partial g$ is monotone, thus ,combining \cref{equ26} and \cref{equ28}, we get
 \begin{equation}\label{equ29}
  (-\lambda^{k+1 })^{T}\frac{\lambda^{k+1}}{\left\lVert \lambda^{k}\right\rVert^{1-\frac{1}{p} } } \geq 0.
 \end{equation} 
 \cref{equ29} implies that $\lambda^{k+1 }=\lambda^{\ast }=0$.
\end{proof}

\begin{remark}
  In the case where $\lambda^{\ast}=0$, the \cref{algorithm2} with the initial point $\lambda^{0} \neq 0$ only takes one iteration to obtain $\lambda^{\ast}$.
\end{remark}

Case 2: $\lambda^{\ast} \neq 0$. The following lemmas are   essential tools for establishing the convergence rate of \cref{algorithm2}.
\begin{lemma} \label{newlemma4}
  Suppose $\lambda^{\ast} \neq 0$ and $ \left\lVert \lambda^{k}\right\rVert >0$. Then, $\lambda^{k+1}$ generated by \cref{algorithm2} satisfies  that
  \begin{equation}\label{equ30b}
    (\lambda^{k+1}-\lambda^{\ast})^{T}\left\{\left\lVert \lambda^{k}\right\rVert^{1-\frac{1}{p} } \lambda^{\ast}-\left\lVert \lambda^{\ast}\right\rVert^{1-\frac{1}{p} }\lambda^{k+1} \right\}  \geq 0,
   \end{equation} 
  \begin{equation}\label{equ30}
    (\left\lVert \lambda^{k}\right\rVert^{1-\frac{1}{p} } -\left\lVert \lambda^{\ast}\right\rVert^{1-\frac{1}{p} })(\lambda^{k+1}-\lambda^{\ast})^{T}\lambda^{k+1} \geq \left\lVert \lambda^{k}\right\rVert ^{1-\frac{1}{p} }\left\lVert \lambda^{k+1}-\lambda^{\ast}\right\rVert ^{2}.
   \end{equation} 
\end{lemma}
\begin{proof}
  Since $\lambda^{\ast }\neq0$, the optimal condition for \cref{equ23} is 
  \begin{equation}\label{equ31}
    (\lambda-\lambda^{\ast})^{T}\left\{h^{\ast}+\mu\frac{\lambda^{\ast}}{\left\lVert \lambda^{\ast}\right\rVert^{1-\frac{1}{p} } } \right\}  \geq 0, \quad \forall \lambda \in \mathbb{R}^{n},
 \end{equation} 
 where $h^{\ast} \in \partial g(\lambda^{\ast})$. Setting $\lambda$ in \cref{equ31} as $\lambda^{k+1}$, we have
 \begin{equation}\label{equ32}
  (\lambda^{k+1}-\lambda^{\ast })^{T}\left\{h^{\ast}+\mu\frac{\lambda^{\ast}}{\left\lVert \lambda^{\ast}\right\rVert^{1-\frac{1}{p} } } \right\} \geq 0.
\end{equation} 
 Using the optimality condition for the step 8 of \cref{algorithm2} yields
 \begin{equation}\label{equ33}
  (\lambda-\lambda^{k+1 })^{T}\left\{h^{k+1}+\mu \frac{\lambda^{k+1}}{\left\lVert \lambda^{k}\right\rVert^{1-\frac{1}{p} } }  \right\} \geq 0, \quad \forall y \in \mathcal{Y}.
 \end{equation} 
  where $h^{k+1} \in \partial g(\lambda^{k+1})$. Setting $\lambda^{\ast }$ in \cref{equ33} as $\lambda$, we have
 \begin{equation}\label{equ34}
  (\lambda^{\ast }-\lambda^{k+1 })^{T}\left\{h^{k+1}+\mu \frac{\lambda^{k+1}}{\left\lVert \lambda^{k}\right\rVert^{1-\frac{1}{p} } }  \right\} \geq 0.
 \end{equation} 
 Considering that $\partial g$ is monotone, thus ,combining \cref{equ32} and \cref{equ34}, we get
 \begin{equation}\label{equ35}
  (\lambda^{k+1}-\lambda^{\ast})^{T}\left\{\left\lVert \lambda^{k}\right\rVert^{1-\frac{1}{p} } \lambda^{\ast}-\left\lVert \lambda^{\ast}\right\rVert^{1-\frac{1}{p} }\lambda^{k+1} \right\}  \geq 0.
 \end{equation} 
 Adding the following term
 \begin{equation}
  \left\lVert \lambda^{k}\right\rVert^{1-\frac{1}{p} }  \left\lVert \lambda^{k+1}-\lambda^{\ast }\right\rVert ^{2}\nonumber
\end{equation} 
to the both sides of \cref{equ35} yields the inequality \cref{equ30}.
\end{proof}
\begin{lemma} \label{newlemma5}
  Suppose $\lambda^{\ast} \neq 0$ and $ \left\lVert \lambda^{k}\right\rVert >0$. Then, $\lambda^{k+1}$ generated by \cref{algorithm2} satisfies  that
  \begin{equation}\label{equ36}
    (\left\lVert \lambda^{k}\right\rVert^{1-\frac{1}{p} }-\left\lVert \lambda^{\ast}\right\rVert^{1-\frac{1}{p} })(\left\lVert \lambda^{k+1}\right\rVert^{2 }-\left\lVert \lambda^{\ast}\right\rVert^{2  }) \geq 0,
   \end{equation} 
   \begin{equation}\label{equ37}
    (\left\lVert \lambda^{k}\right\rVert^{1-\frac{1}{p} }-\left\lVert \lambda^{\ast}\right\rVert^{1-\frac{1}{p} })(\left\lVert \lambda^{k}\right\rVert^{1-\frac{1}{p} }-\left\lVert \lambda^{k+1}\right\rVert^{1-\frac{1}{p}  }) \geq 0.
  \end{equation} 
\end{lemma}
\begin{proof}
  Firstly, we prove the inequality \cref{equ36}. According to \cref{newlemma4}, we have
  \begin{equation}\label{equ38}
    \begin{aligned}
      (\left\lVert \lambda^{k}\right\rVert ^{1-\frac{1}{p} }-\left\lVert \lambda^{\ast}\right\rVert ^{1-\frac{1}{p} })(\lambda^{k+1}-\lambda^{\ast})^{T}\lambda^{k+1}&\geq \left\lVert \lambda^{k}\right\rVert^{1-\frac{1}{p} }  \left\lVert \lambda^{k+1}-\lambda^{\ast}\right\rVert ^{2} \\
      &\geq     (\left\lVert \lambda^{k}\right\rVert ^{1-\frac{1}{p} }-\left\lVert \lambda^{\ast}\right\rVert ^{1-\frac{1}{p} })\left\lVert \lambda^{k+1}-\lambda^{\ast}\right\rVert ^{2}.
    \end{aligned}
  \end{equation} 
  \cref{equ38} implies that 
  \begin{equation}\label{equ39}
    \left\{ {\begin{aligned}
      (\left\lVert \lambda^{k}\right\rVert ^{1-\frac{1}{p} }-\left\lVert \lambda^{\ast}\right\rVert ^{1-\frac{1}{p} })(\lambda^{k+1}-\lambda^{\ast})^{T}\lambda^{k+1}&\geq0\\
      (\left\lVert \lambda^{k}\right\rVert ^{1-\frac{1}{p} }-\left\lVert \lambda^{\ast}\right\rVert ^{1-\frac{1}{p} })(\lambda^{k+1}-\lambda^{\ast})^{T}\lambda^{\ast}&\geq0
      \end{aligned}}. \right.
  \end{equation}
  The assertion \cref{equ36} follows directly from \cref{equ39}.

  Secondly, we prove the inequality \cref{equ37}. It is easy to show that $\lambda^{k+1} \neq 0$. Otherwise, according to \cref{newlemma4}, $\left\lVert \lambda^{k}\right\rVert ^{1-\frac{1}{p} }\left\lVert \lambda^{\ast}\right\rVert ^{2} \leq 0$, which is contradict with the assumption in \cref{newlemma5}. Using the convexity of the function $\left\lVert \cdot \right\rVert ^{1+\frac{1}{p} }$, we have
  \begin{equation}\label{equ40}
    (\lambda^{k+1}-\lambda^{\ast })^{T}\{ \frac{\lambda^{k+1}}{\left\lVert \lambda^{k+1}\right\rVert ^{1-\frac{1}{p} }} -\frac{\lambda^{\ast}}{\left\lVert \lambda^{\ast}\right\rVert ^{1-\frac{1}{p} }}\} \geq 0.
  \end{equation} 
  Combining \cref{equ40} and \cref{equ30b}, it holds that
  \begin{equation}\label{equ41}
    (\left\lVert \lambda^{k}\right\rVert^{1-\frac{1}{p} }-\left\lVert \lambda^{k+1}\right\rVert^{1-\frac{1}{p} })(\lambda^{k+1}-\lambda^{\ast})^{T} \lambda^{\ast}\geq 0.
 \end{equation} 
 Moreover, the inequality
 \begin{equation}\label{equ42}
  (\left\lVert \lambda^{k}\right\rVert^{1-\frac{1}{p} }-\left\lVert \lambda^{\ast}\right\rVert^{1-\frac{1}{p} })(\left\lVert \lambda^{k}\right\rVert^{1-\frac{1}{p} }-\left\lVert \lambda^{k+1}\right\rVert^{1-\frac{1}{p}  }) \left[(\lambda^{k+1}-\lambda^{\ast})^{T} \lambda^{\ast}\right]^{2}\geq 0
\end{equation} 
follows  from \cref{equ39} and \cref{equ41}. Thus, it remains to prove that the assertion \cref{equ37} holds when $(\lambda^{k+1}-\lambda^{\ast})^{T} \lambda^{\ast} = 0$. According to \cref{newlemma4}, $\lambda^{k+1}=\lambda^{\ast}$ when $(\lambda^{k+1}-\lambda^{\ast})^{T} \lambda^{\ast} = 0$, and inequality \cref{equ37} obviously holds.
\end{proof}

\begin{remark}
  From \cref{newlemma5}, when the initial point $\lambda^{0}$ satisfies that $\left\lVert \lambda^{0}\right\rVert  \geq \left\lVert \lambda^{\ast}\right\rVert$, then for $k \geq 0$, we have $\left\lVert \lambda^{k}\right\rVert \geq \left\lVert \lambda^{k+1}\right\rVert$ and $\left\lVert \lambda^{k}\right\rVert \geq \left\lVert \lambda^{\ast}\right\rVert$. Similarly, when the initial point $\lambda^{0}$ satisfies that $\left\lVert \lambda^{0}\right\rVert  \leq \left\lVert \lambda^{\ast}\right\rVert$, then for $k \geq 0$,  $\left\lVert \lambda^{k}\right\rVert \leq \left\lVert \lambda^{k+1}\right\rVert$ and $\left\lVert \lambda^{k}\right\rVert \leq \left\lVert \lambda^{\ast}\right\rVert$. Intuitively, \cref{newlemma5} tells us that $\left\{\lambda^{k}\right\}_{k\geq 0} $ generated by \cref{algorithm2} can make $\left\lVert \lambda^{k}\right\rVert $ approach to $\left\lVert \lambda^{\ast}\right\rVert $.
\end{remark}

Now, we analyze the convergence rate of \cref{algorithm2}. Here, our analysis is under the assumption that $\left\lVert \lambda^{0}\right\rVert  \geq \left\lVert \lambda^{\ast}\right\rVert$, and the proof of the case $\left\lVert \lambda^{0}\right\rVert  \leq \left\lVert \lambda^{\ast}\right\rVert$ can be seen in \cref{appendix}.
\begin{theorem}\label{newTheorem5}
   Let the sequence $\left\{\lambda^{k} \right\}_{1\leq k \leq  N} $ be generated by \cref{algorithm2} with the initial point $\lambda^{0}$. Suppose that $\left\lVert \lambda^{0}\right\rVert  \geq \left\lVert \lambda^{\ast}\right\rVert$, and $\lambda^{k} \neq \lambda^{\ast}, k = 1,2,...,N$. Then, we have 
  \begin{equation}\label{equ43}
     \left\lVert \lambda^{N}-\lambda^{\ast}\right\rVert \leq \left\lVert \lambda^{\ast}\right\rVert (e^{(1-\frac{1}{p} )^{N-1}\ln \frac{\left\lVert \lambda^{0}\right\rVert}{\left\lVert \lambda^{\ast}\right\rVert} }-1)\leq (\left\lVert \lambda^{0}\right\rVert -\left\lVert \lambda^{\ast}\right\rVert)(1-\frac{1}{p} )^{N-1}.
 \end{equation} 
\end{theorem} 
\begin{proof}
  According to
  \cref{newlemma5}, for any $k\geq 0 $, $\left\lVert \lambda^{k}\right\rVert \geq \left\lVert \lambda^{\ast}\right\rVert$. Thus, we have
  \begin{multline} \label{equ44}
    M^{k} = (\left\lVert \lambda^{k}\right\rVert ^{1-\frac{1}{p} }-\left\lVert \lambda^{\ast}\right\rVert ^{1-\frac{1}{p} })(\lambda^{k+1}-\lambda^{\ast})^{T}\lambda^{k+1}\\ \leq (\left\lVert \lambda^{k}\right\rVert ^{1-\frac{1}{p} }-\left\lVert \lambda^{\ast}\right\rVert ^{1-\frac{1}{p} })\left\lVert \lambda^{k+1}-\lambda^{\ast}\right\rVert \left\lVert \lambda^{k+1}\right\rVert.
  \end{multline}
  According to  \cref{newlemma4},  it holds that 
  \begin{equation}\label{equ45}
    \begin{aligned}
      \left\lVert \lambda^{k}\right\rVert^{1-\frac{1}{p} }  \left\lVert \lambda^{k+1}-\lambda^{\ast}\right\rVert ^{2} &\leq  M^{k}\\
      &\leq  (\left\lVert \lambda^{k}\right\rVert ^{1-\frac{1}{p} }-\left\lVert \lambda^{\ast}\right\rVert ^{1-\frac{1}{p} })\left\lVert \lambda^{k+1}-\lambda^{\ast}\right\rVert \left\lVert \lambda^{k+1}\right\rVert.
    \end{aligned}
  \end{equation} 
  Note that $ \left\lVert \lambda^{k+1}-\lambda^{\ast}\right\rVert \geq \left\lVert \lambda^{k+1}\right\rVert-\left\lVert \lambda^{\ast}\right\rVert$, and combining with \cref{equ45}, we have
  \begin{equation}\label{equ46}
    \left\lVert \lambda^{k}\right\rVert ^{1-\frac{1}{p} }(\left\lVert \lambda^{k+1}\right\rVert -\left\lVert \lambda^{\ast}\right\rVert) \leq \left\lVert \lambda^{k+1}\right\rVert (\left\lVert \lambda^{k}\right\rVert^{1-\frac{1}{p} } -\left\lVert \lambda^{\ast}\right\rVert^{1-\frac{1}{p} }).
\end{equation} 
  By a simple manipulation, we obtain
  \begin{equation}\label{equ47}
     \left\lVert \lambda^{k+1}\right\rVert \leq \left\lVert \lambda^{k}\right\rVert^{1-\frac{1}{p} }\left\lVert \lambda^{\ast}\right\rVert^{\frac{1}{p} }.
  \end{equation}  
  Then, 
  \begin{equation}\label{equ48}
    \begin{aligned}
      \ln \frac{\left\lVert \lambda^{k}\right\rVert }{\left\lVert \lambda^{\ast}\right\rVert}  &\leq  (1-\frac{1}{p} )\ln \frac{\left\lVert \lambda^{k-1}\right\rVert }{\left\lVert \lambda^{\ast}\right\rVert}\\
      &\leq  (1-\frac{1}{p} )^{k}\ln \frac{\left\lVert \lambda^{0}\right\rVert }{\left\lVert \lambda^{\ast}\right\rVert}, \quad k = 1,2,...N.
    \end{aligned}
  \end{equation} 
  Note that $\left\lVert \lambda^{k+1}\right\rVert \leq \left\lVert \lambda^{k}\right\rVert$, according to the inequality \cref{equ45}, we have
  \begin{equation}\label{equ49}
    \begin{aligned}
    \left\lVert \lambda^{N}-\lambda^{\ast}\right\rVert   &\leq \left\lVert \lambda^{N-1}\right\rVert ^{\frac{1}{p} } (\left\lVert \lambda^{N-1}\right\rVert^{1-\frac{1}{p} }-\left\lVert \lambda^{\ast}\right\rVert^{1-\frac{1}{p} })\\
      &=  \left\lVert \lambda^{N-1}\right\rVert -\left\lVert \lambda^{N-1}\right\rVert^{\frac{1}{p} }\left\lVert \lambda^{\ast}\right\rVert^{1-\frac{1}{p}}\\
      & \leq \left\lVert \lambda^{\ast}\right\rVert (\frac{\left\lVert \lambda^{N-1}\right\rVert}{\left\lVert \lambda^{\ast}\right\rVert}-1 )\\
      & \leq \left\lVert \lambda^{\ast}\right\rVert (e^{(1-\frac{1}{p} )^{N-1}\ln \frac{\left\lVert \lambda^{0}\right\rVert}{\left\lVert \lambda^{\ast}\right\rVert} }-1)\\
      & \leq (\left\lVert \lambda^{0}\right\rVert -\left\lVert \lambda^{\ast}\right\rVert)(1-\frac{1}{p} )^{N-1}.
    \end{aligned}
  \end{equation} 
\end{proof}

Now, we apply the \cref{algorithm2} to the dual problem \cref{equ20}. The update for $\lambda^{k+1}$  is 
\begin{equation}\label{equ50}
  \lambda ^{k+1} = \arg  \min \{f_1^{\ast}(\lambda)-\lambda^{T}c+ \frac{\sigma^{-\frac{1}{p} } }{2 }t_k\left\lVert \lambda\right\rVert ^{2 } \vert \; \lambda\in\mathbb{R}^{n}   \}.
\end{equation}
Using the analysis results in \cref{sec:dual}, $\lambda^{k+1}$ can be expressed as 
\begin{equation}\label{equ51}
  \lambda^{k+1}=\frac{c-\mathsf{P}\mathrm{rox}_{\sigma^{-\frac{1}{p}} t^{k}f}(c)}{\sigma^{-\frac{1}{p}} t^{k}} .
\end{equation}
The detail procedure for solving the regularization problem \cref{equ1} is presented in \cref{algorithm3}.
\begin{algorithm} 
  \caption{Fixed point iteration for the regularization problem \cref{equ1}}
  \label{algorithm3}
  \begin{algorithmic}[1]
  \STATE{Require: $\lambda^{0} \in \mathbb{R}^{n}, \; K$}
  \WHILE{$k = 0,1,2,...,K$}
     \IF{$\lambda^{k} \neq 0$}
        \STATE $t^{k} = \left\lVert \lambda^{k}\right\rVert ^{\frac{1}{p} -1}$ ; 
     \ELSE
       \STATE $t^{k} =0$;
     \ENDIF 
     \STATE $\sigma^{k} = \sigma^{-\frac{1}{p}} t^{k}$
     \STATE $\lambda^{k+1}=\frac{c-\mathsf{P}\mathrm{rox}_{\sigma^{k}f}(c)}{\sigma^{k}} $
     \STATE $x^{k+1}=c-\sigma^{-\frac{1}{p} }i_p(\lambda^{k+1}) $
  \ENDWHILE
  \RETURN $x^{K}$
  \end{algorithmic}
\end{algorithm}

\section{Special case p=2}
\label{sec:special}
In this section, we consider the regularization problem \cref{equ1} with $p=2$, and we will see that it can be transformed into solving an one-dimensional monotonic continuity equation. Denote by $x^{\ast }$ the optimal solution of \cref{equ1} with $p=2$, and it satisfies the optimal condition
\begin{equation}\label{eq5}
  0 \in  \partial  f(x^{\ast})+\sigma\left\lVert x^{\ast}-c\right\rVert(x^{\ast}-c) ,
\end{equation}
where $\partial  f$ denotes the subgradient of $f$. If $0 \notin \partial  f(c)$, the condition \cref{eq5} can be equivalent to the following condition
\begin{equation}\label{eq6}
  \left\{ {\begin{aligned}
    x^{\ast} &= \mathsf{P}\mathrm{rox}_{\frac{1}{\sigma t}f}\left(c\right)\\
    t &= \left\lVert x^{\ast}-c\right\rVert 
    \end{aligned}} \right..
\end{equation}
The condition \cref{eq6} can be simplified as an equation about $t$, i.e.
\begin{equation}\label{eq7}
  \frac{1}{\sigma}  = \frac{1}{\sigma t} \left\lVert \mathsf{P}\mathrm{rox}_{\frac{1}{\sigma t}f}\left(c\right)-c\right\rVert.
\end{equation}
Define $T(t) = t\left\lVert \mathsf{P}\mathrm{rox}_{tf}\left(c\right)-c\right\rVert $, then we will prove some significant property of $T(t)$.
\begin{lemma}\label{lemma_1}
  If $0<\lambda _1\leq \lambda _2$, then
  \begin{equation}\label{eq8}
    \frac{\lambda _2}{\lambda _1} T(\lambda _1) \leq T(\lambda _2) \leq (\frac{\lambda _2}{\lambda _1})^{2} T(\lambda _1).
  \end{equation}
\end{lemma}
\begin{proof}
  Denote $y_1 = \mathsf{P}\mathrm{rox}_{\lambda _1f}\left(c\right)$ and $y_2 = \mathsf{P}\mathrm{rox}_{\lambda _2f}\left(c\right)$. This implies that
  \begin{equation}\label{eq9}
    \left\{ {\begin{aligned}
      \lambda _1^{-1}(c-y_1) &\in \partial f(y_1)\\
      \lambda _2^{-1}(c-y_2) &\in \partial f(y_2) 
      \end{aligned}} \right..
  \end{equation}
  Due to the monotonicity of $\partial f$, we have
  \begin{equation}\label{eq10}
    (y_1-y_2)^{T}(\lambda _1^{-1}(c-y_1)-\lambda _2^{-1}(c-y_2)) \geq 0.
  \end{equation}
  By a simple manipulation, it follows that 
  \begin{equation}\label{eq11}
    (\lambda _1+\lambda _2)(y_1-c)^{T}(y_2-c) \geq \lambda _2\left\lVert y _1-c\right\rVert^2+ \lambda _1\left\lVert y _2-c\right\rVert^2.
  \end{equation}
  Using the inequality $a^{T}b\leq \left\lVert a\right\rVert \left\lVert b\right\rVert $, we have
  \begin{equation}\label{eq12}
    (\lambda _1+\lambda _2)\left\lVert y_1-c\right\rVert \left\lVert y_2-c\right\rVert \geq \lambda _2\left\lVert y _1-c\right\rVert^2+ \lambda _1\left\lVert y _2-c\right\rVert^2.
  \end{equation}
  \cref{eq12} implies that
  \begin{equation}\label{eq13}
    \left\{ {\begin{aligned}
      \left\lVert y_1-c\right\rVert  &\leq \left\lVert y_2-c\right\rVert\\
      \lambda _1\left\lVert y_2-c\right\rVert &\leq\lambda _2 \left\lVert y_1-c\right\rVert 
      \end{aligned}} \right..
  \end{equation}
  Combining with the definition of $T(t)$, we obtain \cref{eq8}.
\end{proof}

\begin{remark}
  \cref{lemma_1} indicates that $T(t)$ is continuous and monotone on the interval $[0,+\infty)$. Obviously, $T(0)=0$ and $\lim_{t \to \infty} T(t) = +\infty$.
  
\end{remark}

Let $t^{'}= \frac{1}{\sigma t} $, \cref{eq7} is equivalent to the following equation
\begin{equation}\label{eq14}
  \frac{1}{\sigma}  = T(t^{'}).
\end{equation}
The detailed procedure of the proposed method is given in \cref{algo1}.
\begin{algorithm} 
  \caption{Solving the  regularization problem \cref{equ1} with $p=2$ via bisection method}
  \label{algo1}
  \begin{algorithmic}[1]
  \STATE{Require: $\tau_0 = 0, \tau_1 = \tau >0, K$}
  \FOR{$k = 0,1,2,\ldots K $}
  \STATE\label{line1}{$t_k=\frac{\tau _0+\tau _1}{2} $}
  \IF{$T(t_k)>\frac{1}{\sigma} $ }
      \STATE{$\tau_1 = t_k$}
  \ELSE
  \STATE{$\tau_0 = t_k$}
  \ENDIF
  \ENDFOR
  \RETURN{$x_{K}=\mathsf{P}\mathrm{rox}_{t_{K} f}\left(c\right)$}
  \end{algorithmic}
\end{algorithm}

Now we prove the linear convergence rate of \cref{algo1}.
\begin{theorem}\label{Theorem1}
  Let the sequence $\left\{t^{k}\right\}_{1\leq k \leq K}$ is generated by \cref{algo1} with the initial value $\tau _1 = \tau $, and $x^{k} = \mathsf{P}\mathrm{rox}_{t^{k} f}\left(c\right)$. If $T(\tau )>0$ and $x^{k} \neq x^{\ast} (k=1,2,...,K)$,  we have
 \begin{equation}\label{eq16}
   \left\lVert x^{k}-x^{\ast}\right\rVert \leq \left\lVert g^{\ast}\right\rVert \frac{\tau }{2^{k}},
 \end{equation}
 where $g^{\ast} \in \partial f(x^{\ast})$.
\end{theorem}
\begin{proof}
  Combining with the relationship 
  \begin{equation}\label{equ52}
    \left\{ {\begin{aligned}
      x^{k} &= \mathsf{P}\mathrm{rox}_{t^{k} f}\left(c\right)\\
      x^{\ast} &= \mathsf{P}\mathrm{rox}_{t^{\ast} f}\left(c\right)
      \end{aligned}} \right. ,
  \end{equation}
  we have 
  \begin{equation}\label{equ53}
    \left\{ {\begin{aligned}
      c-x^{k} &= t^{k}g^{k}\\
      c-x^{\ast} &= t^{\ast}g^{\ast}
      \end{aligned}} \right.,
  \end{equation}
  where $g^{k} \in \partial f(x^{k})$ and $g^{\ast} \in \partial f(x^{\ast})$. Note that
  \begin{equation}\label{equ54}
    \begin{aligned}
    \left\lVert x^{k}-x^{\ast}\right\rVert ^{2}   &= (x^{\ast}-x^{k})^{T}(x^{\ast}-x^{k})\\
      &=  (t^{k}g^{k}-t^{\ast}g^{\ast})^{T}(x^{\ast}-x^{k})\\
      & =  \left[t^{k}(g^{k}-g^{\ast})+(t^{k}-t^{\ast})g^{\ast}\right] ^{T}(x^{\ast}-x^{k})\\
      & \leq (t^{k}-t^{\ast}){g^{\ast}}^{T}(x^{\ast}-x^{k})\\
      & \leq \left\lvert t^{k}-t^{\ast}\right\rvert \left\lVert g^{\ast}\right\rVert \left\lVert x^{\ast}-x^{k}\right\rVert  .
    \end{aligned}
  \end{equation} 

  Since $T(t)$ is continuous and monotone, it can be easily obtained that
  \begin{equation}\label{eq17}
    \left\lvert t^{k}-t^{\ast}\right\rvert \leq \frac{\tau }{2^{k}} .
  \end{equation}
  Considering the assumption $x^{k} \neq x^{\ast}$, \cref{eq17} and \cref{equ54} implies that \cref{eq16} holds.
\end{proof}

The \cref{algo1}  requires the initial value $\tau_1$ to satisfies 
$T(\tau_1)>0$. The line search technique can be utilized to find the initial value, which is present in \cref{algo2}. The \cref{algo2}  will terminate after a finite number of iterations, because $\lim_{t \to \infty} T(t) = +\infty$.
\begin{algorithm} 
  \caption{Finding the initial value via line search}
  \label{algo2}
  \begin{algorithmic}[1]
  \STATE{Require: $\tau_0 >0$}
  \WHILE{$T(\tau_{k}) >\frac{1}{\sigma} $}
  \STATE{$k = k+1 $}
  \STATE{$\tau_{k} = 2\tau_{k-1} $}
  \ENDWHILE
  \RETURN{$\tau = \tau_{k}$}
  \end{algorithmic}
\end{algorithm}

\section{numerical experiments}
\label{sec:experiments}
In this section, we demonstrate the performance of \cref{algorithm3} with $f(x) = \frac{1}{2} x^{T}Ax+b^{T}x$ ($A\succeq 0$) and $f(x) = \left\lVert x\right\rVert _1$. 

\subsection{Quadratic function}
In this experiment, $f(x)$  is set to a quadratic function $\frac{1}{2} x^{T}Ax+b^{T}x$, where $A = \nabla ^{2}g(c)$, $b=\nabla g(c)$ and $g(x)$ is Log-sum-exp function \cite{nesterov2005smooth}, i.e.
\begin{equation}\label{eq29}
  g(x) = log (\sum_{i = 1}^{m} exp(a_i^{T}x-b_i)) .
\end{equation}
Here, we randomly generate $a_i$, $b_i$ and $c$, where $a_i$, $b_i$ and $c$ follows the normal distribution. The constant $\sigma $ in \cref{equ1} is set to 1, $n=1000$ and $m=2000$. The norm of the gradient of \cref{equ1} (i.e. $\left\lVert Ax+b+\left\lVert x-c\right\rVert^{p-1} (x-c)\right\rVert $)  is used as the measure of the performance of \cref{algorithm3}. The initial point is randomly generated. \cref{fig:fig1} shows the convergence curve for different values of $p$.
\begin{figure}[htbp]
  \centering
  \label{fig:a}\includegraphics[scale=0.3]{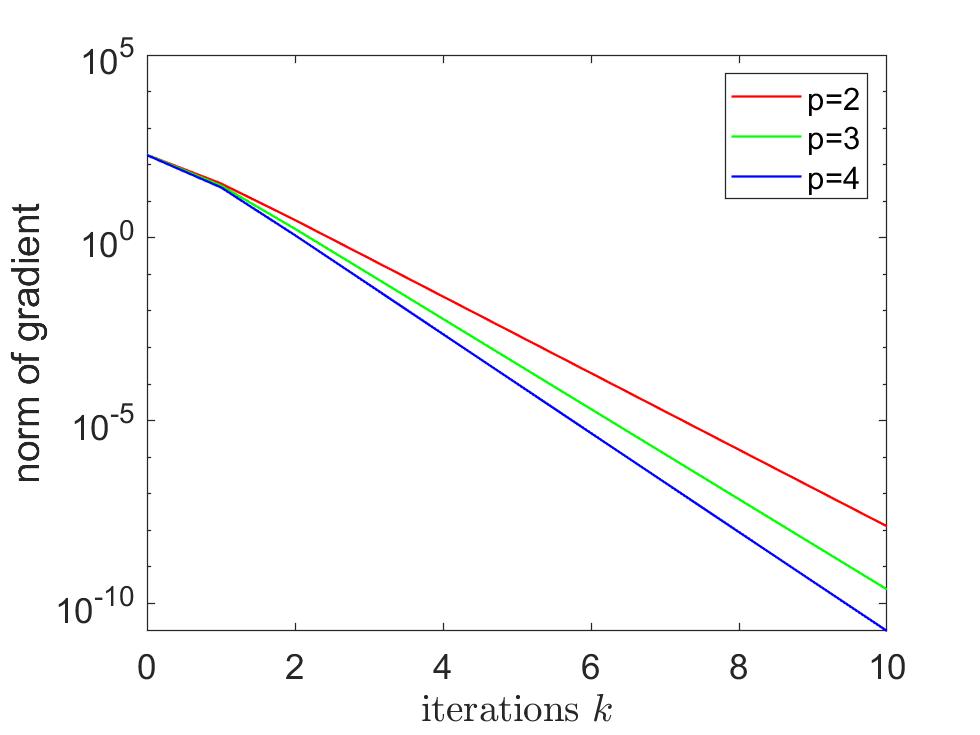}
    \caption{The curve of norms of gradient $Ax+b+\left\lVert x-c\right\rVert^{p-1} (x-c)$ with  $p=2,3,4$.}
  \label{fig:fig1}
\end{figure}

\subsection{ $l_1$ regularization function}
In this experiment, $f(x)$ is set to the $l_1$ regularization function $\left\lVert x\right\rVert _1$, where $\left\lVert x\right\rVert _1 = \sum_{i = 1}^{n} \left\lvert x_i\right\rvert$. The constant $\sigma $ in \cref{equ1} is set to 1, and $n=1000$. In this case, the optimal solution $x^{\ast}$ satisfies the following condition 
\begin{equation}\label{eq30}
  x^{\ast} = \mathsf{P}\mathrm{rox}_{\left\lVert \cdot \right\rVert _1 }\left(x^{\ast}-\left\lVert x^{\ast}-c\right\rVert (x^{\ast}-c)\right).
\end{equation}
Therefore, we define $G(x)=\left\lVert \mathsf{P}\mathrm{rox}_{\left\lVert \cdot \right\rVert _1 }\left(x-\left\lVert x-c\right\rVert (x-c)\right)-x\right\rVert $, and $G(x)$ is used as the measure of the performance of \cref{algorithm3}. The initial point is randomly generated. The convergence curve is shown in \cref{fig:fig2}.
\begin{figure}[htbp]
  \centering
  \label{fig:c}\includegraphics[scale=0.3]{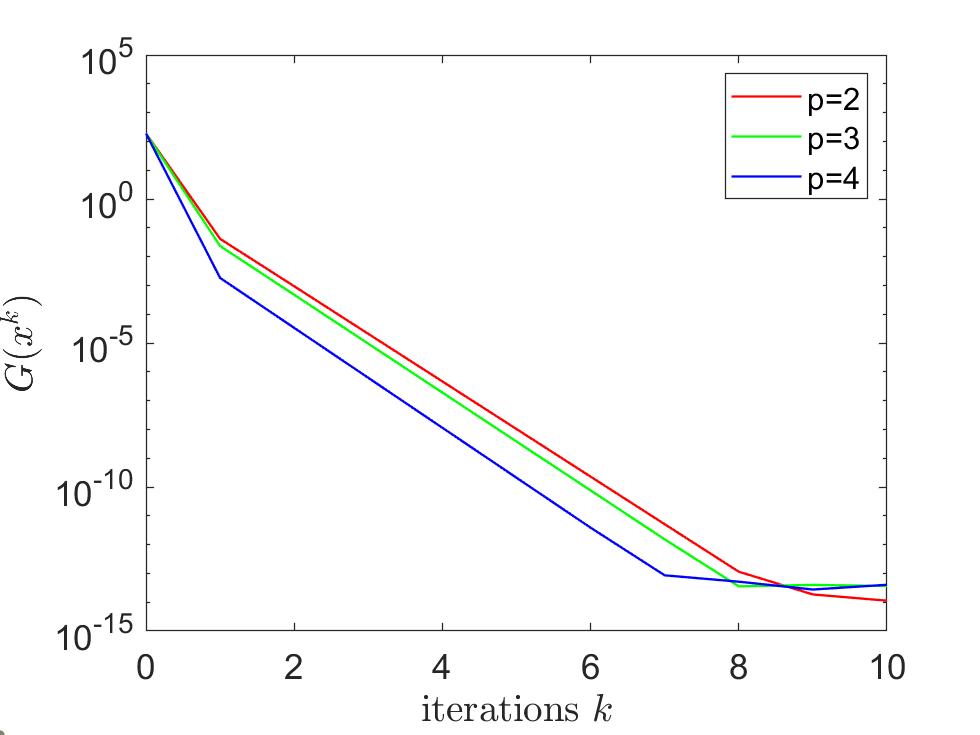}
  \caption{The curve of $G(x_k)$ for  $p=2,3,4$.}
  \label{fig:fig2}
\end{figure}

\section{Conclusions}
\label{sec:conclusions}
In this paper, we propose a linearly convergent method to solve the regularization problem \cref{equ1} based on the classical proximal operator. Moreover, for the special case $p=2$, \cref{equ1} can be transformed into  a one-dimensional monotonic continuity equation. Thus, the bisection method can be used to solve it. This work provides a novel approach to solving the $p$th-order proximal operator ($p>1$).

\appendix
\section{Supplementary proof}
\label{appendix}
\begin{theorem}\label{AppTheorem1}
  Let the sequence $\left\{\lambda^{k} \right\}_{1\leq k \leq  N} $ be generated by \cref{algorithm2} with the initial point $\lambda^{0}$. Suppose that $\left\lVert \lambda^{0}\right\rVert  \leq \left\lVert \lambda^{\ast}\right\rVert$, and $\lambda^{k} \neq \lambda^{\ast}, k = 1,2,...,N$. Then, we have 
 \begin{equation}\label{A1}
    \left\lVert \lambda^{N}-\lambda^{\ast}\right\rVert \leq (1-\frac{1}{p} )\left\lVert \lambda^{\ast}\right\rVert (e^{(1-\frac{1}{p} )^{N-1}\ln \frac{\left\lVert \lambda^{\ast}\right\rVert}{\left\lVert \lambda^{0}\right\rVert} }-1)\leq \frac{\left\lVert \lambda^{\ast}\right\rVert}{\left\lVert \lambda^{0}\right\rVert}  (\left\lVert \lambda^{\ast}\right\rVert -\left\lVert \lambda^{0}\right\rVert)(1-\frac{1}{p} )^{N}.
\end{equation} 
\end{theorem} 
\begin{proof}
  According to
  \cref{newlemma5}, for any $k\geq 0 $, $\left\lVert \lambda^{k}\right\rVert \leq \left\lVert \lambda^{\ast}\right\rVert$. Thus, we have
  \begin{multline} \label{A2}
    M^{k} = (\left\lVert \lambda^{k}\right\rVert ^{1-\frac{1}{p} }-\left\lVert \lambda^{\ast}\right\rVert ^{1-\frac{1}{p} })(\lambda^{k+1}-\lambda^{\ast})^{T}\lambda^{k+1}\\ \leq (\left\lVert \lambda^{\ast}\right\rVert ^{1-\frac{1}{p} }-\left\lVert \lambda^{k}\right\rVert ^{1-\frac{1}{p} })\left\lVert \lambda^{k+1}-\lambda^{\ast}\right\rVert \left\lVert \lambda^{k+1}\right\rVert.
  \end{multline}
  According to  \cref{newlemma4},  it holds that 
  \begin{equation}\label{A3}
    \begin{aligned}
      \left\lVert \lambda^{k}\right\rVert^{1-\frac{1}{p} }  \left\lVert \lambda^{k+1}-\lambda^{\ast}\right\rVert ^{2} &\leq  M^{k}\\
      &\leq  (\left\lVert \lambda^{\ast}\right\rVert ^{1-\frac{1}{p} }-\left\lVert \lambda^{k}\right\rVert ^{1-\frac{1}{p} })\left\lVert \lambda^{k+1}-\lambda^{\ast}\right\rVert \left\lVert \lambda^{k+1}\right\rVert.
    \end{aligned}
  \end{equation} 
  Note that $ \left\lVert \lambda^{k+1}-\lambda^{\ast}\right\rVert \geq \left\lVert \lambda^{\ast}\right\rVert-\left\lVert \lambda^{k+1}\right\rVert$, and combining with \cref{A3}, we have
  \begin{equation}\label{A4}
    \left\lVert \lambda^{k}\right\rVert ^{1-\frac{1}{p} }(\left\lVert \lambda^{\ast}\right\rVert -\left\lVert \lambda^{k+1}\right\rVert) \leq \left\lVert \lambda^{k+1}\right\rVert (\left\lVert \lambda^{\ast}\right\rVert^{1-\frac{1}{p} } -\left\lVert \lambda^{k}\right\rVert^{1-\frac{1}{p} }).
\end{equation} 
 By a simple manipulation, we obtain
  \begin{equation}\label{A5}
    \left\lVert \lambda^{k}\right\rVert^{1-\frac{1}{p} }\left\lVert \lambda^{\ast}\right\rVert^{\frac{1}{p} }  \leq \left\lVert \lambda^{k+1}\right\rVert.
  \end{equation}  
  Then,  
  \begin{equation}\label{A6}
    \begin{aligned}
      \ln \frac{\left\lVert \lambda^{\ast}\right\rVert }{\left\lVert \lambda^{k}\right\rVert}  &\leq  (1-\frac{1}{p} )\ln \frac{\left\lVert \lambda^{\ast}\right\rVert }{\left\lVert \lambda^{k-1}\right\rVert}\\
      &\leq  (1-\frac{1}{p} )^{k}\ln \frac{\left\lVert \lambda^{\ast}\right\rVert }{\left\lVert \lambda^{0}\right\rVert}, \quad k = 1,2,...N.
    \end{aligned}
  \end{equation} 
  Note that $\left\lVert \lambda^{k+1}\right\rVert \leq \left\lVert \lambda^{\ast}\right\rVert$, according to the inequality \cref{equ45}, we have
  \begin{equation}\label{A7}
    \begin{aligned}
    \left\lVert \lambda^{N}-\lambda^{\ast}\right\rVert   &\leq \frac{\left\lVert \lambda^{N}\right\rVert }{\left\lVert \lambda^{N-1}\right\rVert^{1-\frac{1}{p} } }  (\left\lVert \lambda^{\ast}\right\rVert^{1-\frac{1}{p} }-\left\lVert \lambda^{N-1}\right\rVert^{1-\frac{1}{p} })\\
      &\leq   \left\lVert \lambda^{\ast}\right\rVert \left\{(\frac{\left\lVert \lambda^{\ast}\right\rVert}{\left\lVert \lambda^{N-1}\right\rVert})^{1-\frac{1}{p}} -1\right\} \\
      & \leq (1-\frac{1}{p} )\left\lVert \lambda^{\ast}\right\rVert (\frac{\left\lVert \lambda^{\ast}\right\rVert}{\left\lVert \lambda^{N-1}\right\rVert}-1 )\\
      & \leq (1-\frac{1}{p} )\left\lVert \lambda^{\ast}\right\rVert (e^{(1-\frac{1}{p} )^{N-1}\ln \frac{\left\lVert \lambda^{\ast}\right\rVert}{\left\lVert \lambda^{0}\right\rVert} }-1)\\
      & \leq \frac{\left\lVert \lambda^{\ast}\right\rVert}{\left\lVert \lambda^{0}\right\rVert}  (\left\lVert \lambda^{\ast}\right\rVert -\left\lVert \lambda^{0}\right\rVert)(1-\frac{1}{p} )^{N}.
    \end{aligned}
  \end{equation} 
\end{proof}

\bibliographystyle{siamplain}
\bibliography{references}

\begin{thebibliography}{10}

\bibitem{boyd2011distributed}
{\sc S.~Boyd, N.~Parikh, E.~Chu, B.~Peleato, J.~Eckstein, et~al.}, {\em
  Distributed optimization and statistical learning via the alternating
  direction method of multipliers}, Foundations and Trends{\textregistered} in
  Machine learning, 3 (2011), pp.~1--122.

\bibitem{boyd2004convex}
{\sc S.~P. Boyd and L.~Vandenberghe}, {\em Convex optimization}, Cambridge
  university press, 2004.

\bibitem{combettes2007douglas}
{\sc P.~L. Combettes and J.-C. Pesquet}, {\em A douglas--rachford splitting
  approach to nonsmooth convex variational signal recovery}, IEEE Journal of
  Selected Topics in Signal Processing, 1 (2007), pp.~564--574.

\bibitem{hestenes1969multiplier}
{\sc M.~R. Hestenes}, {\em Multiplier and gradient methods}, Journal of
  optimization theory and applications, 4 (1969), pp.~303--320.

\bibitem{hong2017linear}
{\sc M.~Hong and Z.-Q. Luo}, {\em On the linear convergence of the alternating
  direction method of multipliers}, Mathematical Programming, 162 (2017),
  pp.~165--199.

\bibitem{lions1979splitting}
{\sc P.-L. Lions and B.~Mercier}, {\em Splitting algorithms for the sum of two
  nonlinear operators}, SIAM Journal on Numerical Analysis, 16 (1979),
  pp.~964--979.

\bibitem{moreau1965proximite}
{\sc J.-J. Moreau}, {\em Proximit{\'e} et dualit{\'e} dans un espace
  hilbertien}, Bulletin de la Soci{\'e}t{\'e} math{\'e}matique de France, 93
  (1965), pp.~273--299.

\bibitem{nesterov2005smooth}
{\sc Y.~Nesterov}, {\em Smooth minimization of non-smooth functions},
  Mathematical programming, 103 (2005), pp.~127--152.

\bibitem{nesterov2021implementable}
{\sc Y.~Nesterov}, {\em Implementable tensor methods in unconstrained convex
  optimization}, Mathematical Programming, 186 (2021), pp.~157--183.

\bibitem{nesterov2021inexact}
{\sc Y.~Nesterov}, {\em Inexact accelerated high-order proximal-point methods},
  Mathematical Programming,  (2021), pp.~1--26.

\bibitem{nesterov2021inexact2}
{\sc Y.~Nesterov}, {\em Inexact high-order proximal-point methods with
  auxiliary search procedure}, SIAM Journal on Optimization, 31 (2021),
  pp.~2807--2828.

\bibitem{nesterov2022quartic}
{\sc Y.~Nesterov}, {\em Quartic regularity}, arXiv preprint arXiv:2201.04852,
  (2022).

\bibitem{nesterov2006cubic}
{\sc Y.~Nesterov and B.~T. Polyak}, {\em Cubic regularization of newton method
  and its global performance}, Mathematical Programming, 108 (2006),
  pp.~177--205.

\bibitem{powell1969method}
{\sc M.~J. Powell}, {\em A method for nonlinear constraints in minimization
  problems}, Optimization,  (1969), pp.~283--298.

\bibitem{ryu2022large}
{\sc E.~K. Ryu and W.~Yin}, {\em Large-scale convex optimization: algorithms \&
  analyses via monotone operators}, Cambridge University Press, 2022.

\end{thebibliography}
\end{document}